\documentclass{amsart}
\usepackage{amsmath}
\usepackage{amsfonts}

\newtheorem{theorem}{Theorem}
\theoremstyle{plain}

\newtheorem{corollary}{Corollary}

\newtheorem{definition}{Definition}

\newtheorem{lemma}{Lemma}

\newtheorem{proposition}{Proposition}
\newtheorem{remark}{Remark}

\numberwithin{equation}{section}
\input{tcilatex}

\begin{document}
\title[New general integral inequalities]{New general integral inequalities
for Lipschitzian functions via Hadamard fractional integrals}
\author{\.{I}mdat \.{I}\c{s}can}
\address{Department of Mathematics, Faculty of Sciences and Arts, Giresun
University, Giresun, Turkey}
\email{imdat.iscan@giresun.edu.tr, imdati@yahoo.com}

\begin{abstract}
In this paper, the author obtains new estimates on generalization of
Hadamard, Ostrowski and Simpson type inequalities for Lipschitzian functions
via Hadamard fractional integrals. Some applications to special means of
positive reals numbers are also given.
\end{abstract}

\maketitle

\section{Introduction}

Let real function $f$ be defined on some nonempty interval $I$ of real line $%
\mathbb{R}
$. The function $f$ is said to be convex on $I$ if inequality%
\begin{equation}
f(tx+(1-t)y)\leq tf(x)+(1-t)f(y)  \label{1-0}
\end{equation}%
holds for all $x,y\in I$ and $t\in \left[ 0,1\right] .$

Following inequalities are well known in the literature as Hermite-Hadamard
inequality, Ostrowski inequality and Simpson inequality respectively:

\begin{theorem}
Let $f:I\subseteq \mathbb{R\rightarrow R}$ be a convex function defined on
the interval $I$ of real numbers and $a,b\in I$ with $a<b$. The following
double inequality holds%
\begin{equation}
f\left( \frac{a+b}{2}\right) \leq \frac{1}{b-a}\dint\limits_{a}^{b}f(x)dx%
\leq \frac{f(a)+f(b)}{2}\text{.}  \label{1-1}
\end{equation}
\end{theorem}

\begin{theorem}
Let $f:I\subseteq \mathbb{R\rightarrow R}$ be a mapping differentiable in $%
I^{\circ },$ the interior of I, and let $a,b\in I^{\circ }$ with $a<b.$ If $%
\left\vert f^{\prime }(x)\right\vert \leq M,$ $x\in \left[ a,b\right] ,$
then we the following inequality holds%
\begin{equation*}
\left\vert f(x)-\frac{1}{b-a}\dint\limits_{a}^{b}f(t)dt\right\vert \leq 
\frac{M}{b-a}\left[ \frac{\left( x-a\right) ^{2}+\left( b-x\right) ^{2}}{2}%
\right]
\end{equation*}%
for all $x\in \left[ a,b\right] .$
\end{theorem}

\begin{theorem}
Let $f:\left[ a,b\right] \mathbb{\rightarrow R}$ be a four times
continuously differentiable mapping on $\left( a,b\right) $ and $\left\Vert
f^{(4)}\right\Vert _{\infty }=\underset{x\in \left( a,b\right) }{\sup }%
\left\vert f^{(4)}(x)\right\vert <\infty .$ Then the following inequality
holds:%
\begin{equation*}
\left\vert \frac{1}{3}\left[ \frac{f(a)+f(b)}{2}+2f\left( \frac{a+b}{2}%
\right) \right] -\frac{1}{b-a}\dint\limits_{a}^{b}f(x)dx\right\vert \leq 
\frac{1}{2880}\left\Vert f^{(4)}\right\Vert _{\infty }\left( b-a\right) ^{4}.
\end{equation*}
\end{theorem}

In recent years, many athors have studied errors estimations for
Hermite-Hadamard, Ostrowski and Simpson inequalities; for refinements,
counterparts, generalization see \cite%
{ADDC10,AKO11,I13,I13b,I13c,SOS12,SSO10} and references therein.

The following definitions are well known in the literature.

\begin{definition}
A function $f:I\subseteq \mathbb{R\rightarrow R}$ is called an $M$%
-Lipschitzian function on the interval $I$ of real numbers with $M\geq 0$,
if 
\begin{equation*}
\left\vert f(x)-f(y)\right\vert \leq M\left\vert x-y\right\vert
\end{equation*}%
for all $x,y\in I.$
\end{definition}

For some recent results connected with Hermite-Hadamard type integral
inequalities for Lipschitzian functions, see \cite{DCK00,HHT13,THH12,YT01}.

\begin{definition}[\protect\cite{N00,N03}]
A function $f:I\subseteq \left( 0,\infty \right) \rightarrow 
\mathbb{R}
$ is said to be GA-convex (geometric-arithmatically convex) if%
\begin{equation*}
f(x^{t}y^{1-t})\leq tf(x)+\left( 1-t\right) f(y)
\end{equation*}%
for all $x,y\in I$ and $t\in \left[ 0,1\right] $.
\end{definition}

We will now give definitions of the right-sided and left-sided Hadamard
fractional integrals which are used throughout this paper.

\begin{definition}
Let $f\in L\left[ a,b\right] $. The right-sided and left-sided Hadamard
fractional integrals $J_{a^{+}}^{\alpha }f$ and $J_{b^{-}}^{\alpha }f$ of
oder $\alpha >0$ with $b>a\geq 0$ are defined by

\begin{equation*}
J_{a+}^{\alpha }f(x)=\frac{1}{\Gamma (\alpha )}\dint\limits_{a}^{x}\left(
\ln \frac{x}{t}\right) ^{\alpha -1}f(t)\frac{dt}{t},\ a<x<b
\end{equation*}

and

\begin{equation*}
J_{b-}^{\alpha }f(x)=\frac{1}{\Gamma (\alpha )}\dint\limits_{x}^{b}\left(
\ln \frac{t}{x}\right) ^{\alpha -1}f(t)\frac{dt}{t},\ a<x<b
\end{equation*}%
respectively, where $\Gamma (\alpha )$ is the Gamma function defined by $%
\Gamma (\alpha )=$ $\dint\limits_{0}^{\infty }e^{-t}t^{\alpha -1}dt$ (see 
\cite{KST06}).
\end{definition}

In \cite{I13d}, Iscan established Hermite-Hadamard's inequalities for
GA-convex functions in Hadamard fractional integral forms as follows.

\begin{theorem}
Let $f:I\subseteq \left( 0,\infty \right) \rightarrow 
\mathbb{R}
$ be a function such that $f\in L[a,b]$, where $a,b\in I$ with $a<b$. If $f$
is a GA-convex function on $[a,b]$, then the following inequalities for
fractional integrals hold:%
\begin{equation}
f\left( \sqrt{ab}\right) \leq \frac{\Gamma (\alpha +1)}{2\left( \ln \frac{b}{%
a}\right) ^{\alpha }}\left\{ J_{a+}^{\alpha }f(b)+J_{b-}^{\alpha
}f(a)\right\} \leq \frac{f(a)+f(b)}{2}  \label{1-2}
\end{equation}%
with $\alpha >0$.
\end{theorem}

In the inequality (\ref{1-2}), if we take $\alpha =1,$ then we have the
following inequality%
\begin{equation}
f\left( \sqrt{ab}\right) \leq \frac{1}{\ln b-\ln a}\dint\limits_{a}^{b}\frac{%
f(t)}{t}dt\leq \frac{f(a)+f(b)}{2}.  \label{1-3}
\end{equation}

\bigskip Morever in \cite{I13d}, Iscan obtained a generalization of
Hadamard, Ostrowski and Simpson type inequalities for quasi-geometrically
convex functions via Hadamard fractional integrals as related the inequality
(\ref{1-2}).

In this paper, the author obtains new general inequalities for Lipschitzian
functions via Hadamard fractional integrals as related the inequality (\ref%
{1-2}).

\ 

\section{Main Results}

Let $f:I\subseteq \left( 0,\infty \right) \rightarrow 
\mathbb{R}
$ be a differentiable function on $I^{\circ }$, the interior of $I$,
throughout this section we will take%
\begin{eqnarray*}
&&I_{f}\left( x,\lambda ,\alpha ,a,b\right) =\left( 1-\lambda \right) \left[
\ln ^{\alpha }\frac{x}{a}+\ln ^{\alpha }\frac{b}{x}\right] f(x) \\
&&+\lambda \left[ f(a)\ln ^{\alpha }\frac{x}{a}+f(b)\ln ^{\alpha }\frac{b}{x}%
\right] -\Gamma \left( \alpha +1\right) \left[ J_{x-}^{\alpha
}f(a)+J_{x+}^{\alpha }f(b)\right]
\end{eqnarray*}%
and 
\begin{eqnarray*}
&&S_{f}\left( x,y,\lambda ,\alpha ,a,b\right) \\
&=&\lambda ^{\alpha }f(x)+\left( 1-\lambda \right) ^{\alpha }f(y)-\frac{%
\Gamma \left( \alpha +1\right) }{\ln ^{\alpha }\frac{b}{a}}\left[
J_{C-}^{\alpha }f(a)+J_{C+}^{\alpha }f(b)\right]
\end{eqnarray*}%
where $a,b\in I$ with $a<b$, $\ x,y\in \lbrack a,b]$ , $\lambda \in \left[
0,1\right] $, $C=a^{1-\lambda }b^{\lambda }$, $\alpha >0$ and $\Gamma $ is
Euler Gamma function.

\begin{theorem}
\label{2.1}Let $f:$ $I\subseteq \left( 0,\infty \right) \rightarrow 
\mathbb{R}
$ be a $M$-Lipschitzian function on $I$ and $a,b\in I$ with $a<b$. then for
all$\ x\in \lbrack a,b]$ , $\lambda \in \left[ 0,1\right] $ and $\alpha >0$
we have the following inequality for Hadamard fractional integrals%
\begin{eqnarray*}
&&\left\vert I_{f}\left( x,\lambda ,\alpha ,a,b\right) \right\vert \\
&\leq &M\left\{ \left[ \left( 1-\lambda \right) x-\lambda a\right] \left(
\ln \frac{x}{a}\right) ^{\alpha }+\alpha \left( 2\lambda -1\right)
\dint\limits_{a}^{x}\left( \ln \frac{t}{a}\right) ^{\alpha -1}dt\right. \\
&&+\left. \left[ \lambda b-\left( 1-\lambda \right) x\right] \left( \ln 
\frac{b}{x}\right) ^{\alpha }+\alpha \left( 1-2\lambda \right)
\dint\limits_{x}^{b}\left( \ln \frac{b}{t}\right) ^{\alpha -1}dt\right\} .
\end{eqnarray*}
\end{theorem}

\begin{proof}
Using the hypothesis of $f$ , we have the following inequality%
\begin{eqnarray*}
&&\left\vert I_{f}\left( x,\lambda ,\alpha ,a,b\right) \right\vert
=\left\vert \left( 1-\lambda \right) \left[ \ln ^{\alpha }\frac{x}{a}+\ln
^{\alpha }\frac{b}{x}\right] f(x)\right. \\
&&\left. +\lambda \left[ f(a)\ln ^{\alpha }\frac{x}{a}+f(b)\ln ^{\alpha }%
\frac{b}{x}\right] -\alpha \left[ \dint\limits_{a}^{x}\left( \ln \frac{t}{a}%
\right) ^{\alpha -1}f(t)\frac{dt}{t}+\dint\limits_{x}^{b}\left( \ln \frac{b}{%
t}\right) ^{\alpha -1}f(t)\frac{dt}{t}\right] \right\vert
\end{eqnarray*}%
\begin{eqnarray*}
&\leq &\left( 1-\lambda \right) \left\vert f(x)\ln ^{\alpha }\frac{x}{a}%
-\alpha \dint\limits_{a}^{x}\left( \ln \frac{t}{a}\right) ^{\alpha -1}f(t)%
\frac{dt}{t}+f(x)\ln ^{\alpha }\frac{b}{x}-\alpha \dint\limits_{x}^{b}\left(
\ln \frac{b}{t}\right) ^{\alpha -1}f(t)\frac{dt}{t}\right\vert \\
&&+\lambda \left\vert f(a)\ln ^{\alpha }\frac{x}{a}-\alpha
\dint\limits_{a}^{x}\left( \ln \frac{t}{a}\right) ^{\alpha -1}f(t)\frac{dt}{t%
}+f(b)\ln ^{\alpha }\frac{b}{x}-\alpha \dint\limits_{x}^{b}\left( \ln \frac{b%
}{t}\right) ^{\alpha -1}f(t)\frac{dt}{t}\right\vert
\end{eqnarray*}%
\begin{eqnarray*}
&\leq &\alpha \left( 1-\lambda \right) \left[ \dint\limits_{a}^{x}\left( \ln 
\frac{t}{a}\right) ^{\alpha -1}\left\vert f(x)-f(t)\right\vert \frac{dt}{t}%
+\dint\limits_{x}^{b}\left( \ln \frac{b}{t}\right) ^{\alpha -1}\left\vert
f(x)-f(t)\right\vert \frac{dt}{t}\right] \\
&&+\alpha \lambda \left[ \dint\limits_{a}^{x}\left( \ln \frac{t}{a}\right)
^{\alpha -1}\left\vert f(a)-f(t)\right\vert \frac{dt}{t}+\dint%
\limits_{x}^{b}\left( \ln \frac{b}{t}\right) ^{\alpha -1}\left\vert
f(b)-f(t)\right\vert \frac{dt}{t}\right]
\end{eqnarray*}%
\begin{eqnarray*}
&\leq &\alpha \left( 1-\lambda \right) M\left[ \dint\limits_{a}^{x}\left(
\ln \frac{t}{a}\right) ^{\alpha -1}\left( x-t\right) \frac{dt}{t}%
+\dint\limits_{x}^{b}\left( \ln \frac{b}{t}\right) ^{\alpha -1}\left(
t-x\right) \frac{dt}{t}\right] \\
&&+\alpha \lambda M\left[ \dint\limits_{a}^{x}\left( \ln \frac{t}{a}\right)
^{\alpha -1}\left( t-a\right) \frac{dt}{t}+\dint\limits_{x}^{b}\left( \ln 
\frac{b}{t}\right) ^{\alpha -1}\left( b-t\right) \frac{dt}{t}\right]
\end{eqnarray*}%
\begin{eqnarray*}
&\leq &M\left\{ \left[ \left( 1-\lambda \right) x-\lambda a\right] \left(
\ln \frac{x}{a}\right) ^{\alpha }+\alpha \left( 2\lambda -1\right)
\dint\limits_{a}^{x}\left( \ln \frac{t}{a}\right) ^{\alpha -1}dt\right. \\
&&+\left. \left[ \lambda b-\left( 1-\lambda \right) x\right] \left( \ln 
\frac{b}{x}\right) ^{\alpha }+\alpha \left( 1-2\lambda \right)
\dint\limits_{x}^{b}\left( \ln \frac{b}{t}\right) ^{\alpha -1}dt\right\} .
\end{eqnarray*}
\end{proof}

\begin{corollary}
In Theorem \ref{2.1}, If we take $\lambda =0,$ then we get%
\begin{eqnarray}
&&\left\vert \frac{\left( \ln \frac{x}{a}\right) ^{\alpha }+\left( \ln \frac{%
b}{x}\right) ^{\alpha }}{\left( \ln \frac{b}{a}\right) ^{\alpha }}f(x)-\frac{%
\Gamma \left( \alpha +1\right) }{\left( \ln \frac{b}{a}\right) ^{\alpha }}%
\left[ J_{x-}^{\alpha }f(a)+J_{x+}^{\alpha }f(b)\right] \right\vert
\label{3-10} \\
&\leq &\frac{M}{\left( \ln \frac{b}{a}\right) ^{\alpha }}\left\{ x\left[
\left( \ln \frac{x}{a}\right) ^{\alpha }-\left( \ln \frac{b}{x}\right)
^{\alpha }\right] +\alpha \left[ \dint\limits_{x}^{b}\left( \ln \frac{b}{t}%
\right) ^{\alpha -1}dt-\dint\limits_{a}^{x}\left( \ln \frac{t}{a}\right)
^{\alpha -1}dt\right] \right\} .  \notag
\end{eqnarray}%
In this inequality,

\begin{enumerate}
\item[(i)] If we take $\alpha =1$, then 
\begin{eqnarray*}
&&\left\vert f(x)-\frac{1}{\ln b-\ln a}\dint\limits_{a}^{b}\frac{f(t)}{t}%
dt\right\vert \\
&\leq &\frac{M}{\ln b-\ln a}\left\{ x\ln \frac{x^{2}}{ab}+\left(
a+b-2x\right) \right\} .
\end{eqnarray*}

\item[(ii)] If we take $x=\sqrt{ab}$, then 
\begin{eqnarray*}
&&\left\vert f(\sqrt{ab})-\frac{2^{\alpha -1}\Gamma \left( \alpha +1\right) 
}{\left( \ln \frac{b}{a}\right) ^{\alpha }}\left[ J_{\sqrt{ab}-}^{\alpha
}f(a)+J_{\sqrt{ab}+}^{\alpha }f(b)\right] \right\vert \\
&\leq &\frac{2^{\alpha -1}M\alpha }{\left( \ln \frac{b}{a}\right) ^{\alpha }}%
\left\{ \dint\limits_{\sqrt{ab}}^{b}\left( \ln \frac{b}{t}\right) ^{\alpha
-1}dt-\dint\limits_{a}^{\sqrt{ab}}\left( \ln \frac{t}{a}\right) ^{\alpha
-1}dt\right\} .
\end{eqnarray*}

\item[(iii)] If we take $x=\sqrt{ab}$ and~$\alpha =1$ then%
\begin{eqnarray}
&&\left\vert f(\sqrt{ab})-\frac{1}{\ln b-\ln a}\dint\limits_{a}^{b}\frac{f(t)%
}{t}dt\right\vert  \label{3-1} \\
&\leq &\frac{M}{\ln b-\ln a}\left( a+b-2\sqrt{ab}\right) .  \notag
\end{eqnarray}
\end{enumerate}
\end{corollary}

\begin{corollary}
In Theorem \ref{2.1}, If we take $\lambda =1,$ then we get%
\begin{eqnarray*}
&&\left\vert \left[ \frac{f(a)\ln ^{\alpha }\frac{x}{a}+f(b)\ln ^{\alpha }%
\frac{b}{x}}{\left( \ln \frac{b}{a}\right) ^{\alpha }}\right] -\frac{%
2^{\alpha -1}\Gamma \left( \alpha +1\right) }{\left( \ln \frac{b}{a}\right)
^{\alpha }}\left[ J_{x-}^{\alpha }f(a)+J_{x+}^{\alpha }f(b)\right]
\right\vert \\
&\leq &\frac{M}{\left( \ln \frac{b}{a}\right) ^{\alpha }}\left\{ b\ln
^{\alpha }\frac{x}{a}-a\ln ^{\alpha }\frac{b}{x}-\alpha \left[
\dint\limits_{x}^{b}\left( \ln \frac{b}{t}\right) ^{\alpha
-1}dt-\dint\limits_{a}^{x}\left( \ln \frac{t}{a}\right) ^{\alpha -1}dt\right]
\right\} .
\end{eqnarray*}%
In this inequality, if we take $x=\sqrt{ab}$, then 
\begin{eqnarray}
&&\left\vert \frac{f(a)+f(b)}{2}-\frac{2^{\alpha -1}\Gamma \left( \alpha
+1\right) }{\left( \ln \frac{b}{a}\right) ^{\alpha }}\left[ J_{\sqrt{ab}%
-}^{\alpha }f(a)+J_{\sqrt{ab}+}^{\alpha }f(b)\right] \right\vert
\label{3-20} \\
&\leq &\frac{2^{\alpha -1}M}{\left( \ln \frac{b}{a}\right) ^{\alpha }}%
\left\{ \frac{b-a}{2^{\alpha }}\left( \ln \frac{b}{a}\right) ^{\alpha
}-\alpha \left[ \dint\limits_{\sqrt{ab}}^{b}\left( \ln \frac{b}{t}\right)
^{\alpha -1}dt-\dint\limits_{a}^{\sqrt{ab}}\left( \ln \frac{t}{a}\right)
^{\alpha -1}dt\right] \right\} .  \notag
\end{eqnarray}

Specially if we take $\alpha =1$ in this inequality, then we have%
\begin{eqnarray}
&&\left\vert \frac{f(a)+f(b)}{2}-\frac{1}{\ln b-\ln a}\dint\limits_{a}^{b}%
\frac{f(t)}{t}dt\right\vert  \label{3-2} \\
&\leq &\frac{M}{\ln b-\ln a}\left\{ \frac{b-a}{2}\left( \ln \frac{b}{a}%
\right) -\left( a+b-2\sqrt{ab}\right) \right\} .  \notag
\end{eqnarray}
\end{corollary}

\begin{corollary}
In Theorem \ref{2.1},

\begin{enumerate}
\item If we take $x=\sqrt{ab}$ and $\lambda =1/3,$ then 
\begin{eqnarray*}
&&\left\vert \frac{1}{3}\left[ \frac{f(a)+f(b)}{2}+2f\left( \sqrt{ab}\right) %
\right] -\frac{2^{\alpha -1}\Gamma \left( \alpha +1\right) }{\left( \ln 
\frac{b}{a}\right) ^{\alpha }}\left[ J_{\sqrt{ab}-}^{\alpha }f(a)+J_{\sqrt{ab%
}+}^{\alpha }f(b)\right] \right\vert \\
&\leq &\frac{2^{\alpha -1}M}{3\left( \ln \frac{b}{a}\right) ^{\alpha }}%
\left\{ \frac{b-a}{2^{\alpha }}\left( \ln \frac{b}{a}\right) ^{\alpha
}-\alpha \left[ \dint\limits_{\sqrt{ab}}^{b}\left( \ln \frac{b}{t}\right)
^{\alpha -1}dt-\dint\limits_{a}^{\sqrt{ab}}\left( \ln \frac{t}{a}\right)
^{\alpha -1}dt\right] \right\} .
\end{eqnarray*}%
Specially if we take $\alpha =1$ in this inequality, then we have%
\begin{eqnarray}
&&\left\vert \frac{1}{3}\left[ \frac{f(a)+f(b)}{2}+2f\left( \sqrt{ab}\right) %
\right] -\frac{1}{\ln b-\ln a}\dint\limits_{a}^{b}\frac{f(t)}{t}dt\right\vert
\label{3-3} \\
&\leq &\frac{M}{3\left( \ln b-\ln a\right) }\left\{ \frac{b-a}{2}\left( \ln 
\frac{b}{a}\right) -\left( a+b-2\sqrt{ab}\right) \right\} .  \notag
\end{eqnarray}

\item If we take $x=\sqrt{ab}$ and $\lambda =1/2,$ then 
\begin{eqnarray*}
&&\left\vert \frac{1}{2}\left[ \frac{f(a)+f(b)}{2}+2f\left( \sqrt{ab}\right) %
\right] -\frac{2^{\alpha -1}\Gamma \left( \alpha +1\right) }{\left( \ln 
\frac{b}{a}\right) ^{\alpha }}\left[ J_{\sqrt{ab}-}^{\alpha }f(a)+J_{\sqrt{ab%
}+}^{\alpha }f(b)\right] \right\vert \\
&\leq &\frac{M\left( b-a\right) }{4}.
\end{eqnarray*}%
Specially if we take $\alpha =1$ in this inequality, then we have%
\begin{equation}
\left\vert \frac{1}{2}\left[ \frac{f(a)+f(b)}{2}+2f\left( \sqrt{ab}\right) %
\right] -\frac{1}{\ln b-\ln a}\dint\limits_{a}^{b}\frac{f(t)}{t}%
dt\right\vert \leq \frac{M\left( b-a\right) }{4}.  \label{3-4}
\end{equation}
\end{enumerate}
\end{corollary}

\begin{corollary}
\bigskip In Theorem \ref{2.1}, If we take $\alpha =1$, then%
\begin{eqnarray*}
&&\left\vert \left( 1-\lambda \right) f(x)+\lambda \left[ \frac{f(a)\ln 
\frac{x}{a}+f(b)\ln \frac{b}{x}}{\ln b-\ln a}\right] -\frac{1}{\ln b-\ln a}%
\dint\limits_{a}^{b}\frac{f(t)}{t}dt\right\vert \\
&\leq &\frac{M}{\ln b-\ln a}\left\{ \left[ \left( 1-\lambda \right)
x-\lambda a\right] \left( \ln \frac{x}{a}\right) +\left[ \lambda b-\left(
1-\lambda \right) x\right] \left( \ln \frac{b}{x}\right) \right. \\
&&+\left. \left( 1-2\lambda \right) \left( a+b-2x\right) \right\} .
\end{eqnarray*}%
Specially if we take $x=\sqrt{ab}$ in this inequality, then we have%
\begin{eqnarray}
&&\left\vert \left( 1-\lambda \right) f\left( \sqrt{ab}\right) +\lambda
\left( \frac{f(a)+f(b)}{2}\right) -\frac{1}{\ln b-\ln a}\dint\limits_{a}^{b}%
\frac{f(t)}{t}dt\right\vert  \label{3-5} \\
&\leq &\frac{M}{\ln b-\ln a}\left[ \frac{\lambda \left( b-a\right) }{2}\ln 
\frac{b}{a}+\left( 1-2\lambda \right) \left( a+b-2\sqrt{ab}\right) \right] .
\notag
\end{eqnarray}%
We note that if we take $\lambda =0$, $\lambda =1$, $\lambda =1/3$ and $%
\lambda =1/2$ in the inequality (\ref{3-5}) we obtain the inequalities (\ref%
{3-1}), (\ref{3-2}), (\ref{3-3}) and (\ref{3-4}) respectively.
\end{corollary}

Let $f:I\subseteq \left( 0,\infty \right) \rightarrow 
\mathbb{R}
$ be an $M$-Lipschitzian function. In the next theorem, let $\lambda \in %
\left[ 0,1\right] $, $C=a^{1-\lambda }b^{\lambda }$, $x,y\in \lbrack a,b]$
and define $C_{\alpha ,\lambda }$, $\alpha >0$, as follows:

(1) If $a\leq C\leq x\leq y\leq b$, then%
\begin{eqnarray*}
C_{\alpha ,\lambda }(x,y) &=&\frac{x}{\alpha }\lambda ^{\alpha }\left( \ln 
\frac{b}{a}\right) ^{\alpha }+\frac{y}{\alpha }\left\{ \left( 1-\lambda
\right) ^{\alpha }\left( \ln \frac{b}{a}\right) ^{\alpha }-2\left( \ln \frac{%
b}{y}\right) ^{\alpha }\right\} \\
&&+\dint\limits_{y}^{b}\left( \ln \frac{b}{t}\right) ^{\alpha
-1}dt-\dint\limits_{C}^{y}\left( \ln \frac{b}{t}\right) ^{\alpha
-1}dt-\dint\limits_{a}^{C}\left( \ln \frac{t}{a}\right) ^{\alpha -1}dt
\end{eqnarray*}

(2) If $a\leq x\leq C\leq y\leq b$, then%
\begin{eqnarray*}
C_{\alpha ,\lambda }(x,y) &=&\frac{x}{\alpha }\left\{ 2\left( \ln \frac{x}{a}%
\right) ^{\alpha }-\lambda ^{\alpha }\left( \ln \frac{b}{a}\right) ^{\alpha
}\right\} +\frac{y}{\alpha }\left\{ \left( 1-\lambda \right) ^{\alpha
}\left( \ln \frac{b}{a}\right) ^{\alpha }-2\left( \ln \frac{b}{y}\right)
^{\alpha }\right\} \\
&&+\dint\limits_{x}^{C}\left( \ln \frac{t}{a}\right) ^{\alpha
-1}dt-\dint\limits_{a}^{x}\left( \ln \frac{t}{a}\right) ^{\alpha
-1}dt+\dint\limits_{y}^{b}\left( \ln \frac{b}{t}\right) ^{\alpha
-1}dt-\dint\limits_{C}^{y}\left( \ln \frac{b}{t}\right) ^{\alpha -1}dt.
\end{eqnarray*}

(3) If $a\leq x\leq y\leq C\leq b$, then%
\begin{eqnarray*}
C_{\alpha ,\lambda }(x,y) &=&\frac{x}{\alpha }\left\{ 2\left( \ln \frac{x}{a}%
\right) ^{\alpha }-\lambda ^{\alpha }\left( \ln \frac{b}{a}\right) ^{\alpha
}\right\} -\frac{y}{\alpha }\left( 1-\lambda \right) ^{\alpha }\left( \ln 
\frac{b}{a}\right) ^{\alpha } \\
&&+\dint\limits_{C}^{b}\left( \ln \frac{b}{t}\right) ^{\alpha
-1}dt+\dint\limits_{x}^{C}\left( \ln \frac{t}{a}\right) ^{\alpha
-1}dt-\dint\limits_{a}^{x}\left( \ln \frac{t}{a}\right) ^{\alpha -1}dt.
\end{eqnarray*}%
Now we shall give another result for Lipschitzian functions as follows.

\begin{theorem}
\label{2.2}Let $x,y,\alpha ,\lambda ,C,C_{\alpha ,\lambda }$ and the
function $f$ be defined as above. Then we have the following inequality for
hadamard fractional integrals%
\begin{equation}
\left\vert S_{f}\left( x,y,\lambda ,\alpha ,a,b\right) \right\vert \leq 
\frac{\alpha MC_{\alpha ,\lambda }(x,y)}{\left( \ln \frac{b}{a}\right)
^{\alpha }}  \label{2-2}
\end{equation}
\end{theorem}

\begin{proof}
Using the hypothesis of $f$, we have the following inequality%
\begin{eqnarray}
&&\left\vert S_{f}\left( x,y,\lambda ,\alpha ,a,b\right) \right\vert  \notag
\\
&=&\frac{\alpha }{\left( \ln \frac{b}{a}\right) ^{\alpha }}\left\vert
\dint\limits_{a}^{C}\frac{\left[ f(x)-f(t)\right] }{t}\left( \ln \frac{t}{a}%
\right) ^{\alpha -1}dt+\dint\limits_{C}^{b}\frac{\left[ f(y)-f(t)\right] }{t}%
\left( \ln \frac{b}{t}\right) ^{\alpha -1}dt\right\vert  \notag
\end{eqnarray}%
\begin{eqnarray}
&\leq &\frac{\alpha }{\left( \ln \frac{b}{a}\right) ^{\alpha }}\left[
\dint\limits_{a}^{C}\frac{\left\vert f(x)-f(t)\right\vert }{t}\left( \ln 
\frac{t}{a}\right) ^{\alpha -1}dt+\dint\limits_{C}^{b}\frac{f\left\vert
(y)-f(t)\right\vert }{t}\left( \ln \frac{b}{t}\right) ^{\alpha -1}dt\right] 
\notag \\
&\leq &\frac{\alpha M}{\left( \ln \frac{b}{a}\right) ^{\alpha }}\left[
\dint\limits_{a}^{C}\frac{\left\vert x-t\right\vert }{t}\left( \ln \frac{t}{a%
}\right) ^{\alpha -1}dt+\dint\limits_{C}^{b}\frac{\left\vert y-t\right\vert 
}{t}\left( \ln \frac{b}{t}\right) ^{\alpha -1}dt\right]  \label{2-2a}
\end{eqnarray}%
Now using simple calculations, we obtain the following identities $%
\int_{a}^{C}\frac{\left\vert x-t\right\vert }{t}\left( \ln \frac{t}{a}%
\right) ^{\alpha -1}dt$ and $\int_{V}^{b}\frac{\left\vert y-t\right\vert }{t}%
\left( \ln \frac{b}{t}\right) ^{\alpha -1}dt.$

(1) If $a\leq C\leq x\leq y\leq b$, then%
\begin{equation*}
\dint\limits_{a}^{C}\frac{\left\vert x-t\right\vert }{t}\left( \ln \frac{t}{a%
}\right) ^{\alpha -1}dt=\frac{x}{\alpha }\lambda ^{\alpha }\left( \ln \frac{b%
}{a}\right) ^{\alpha }-\dint\limits_{a}^{C}\left( \ln \frac{t}{a}\right)
^{\alpha -1}dt
\end{equation*}%
and%
\begin{eqnarray*}
&&\dint\limits_{C}^{b}\frac{\left\vert y-t\right\vert }{t}\left( \ln \frac{b%
}{t}\right) ^{\alpha -1}dt \\
&=&\frac{y}{\alpha }\left\{ \left( 1-\lambda \right) ^{\alpha }\left( \ln 
\frac{b}{a}\right) ^{\alpha }-2\left( \ln \frac{b}{y}\right) ^{\alpha
}\right\} +\dint\limits_{y}^{b}\left( \ln \frac{b}{t}\right) ^{\alpha
-1}dt-\dint\limits_{C}^{y}\left( \ln \frac{b}{t}\right) ^{\alpha -1}dt.
\end{eqnarray*}%
(2) If $a\leq x\leq C\leq y\leq b$, then%
\begin{eqnarray*}
&&\dint\limits_{a}^{C}\frac{\left\vert x-t\right\vert }{t}\left( \ln \frac{t%
}{a}\right) ^{\alpha -1}dt \\
&=&\frac{x}{\alpha }\left\{ 2\left( \ln \frac{x}{a}\right) ^{\alpha
}-\lambda ^{\alpha }\left( \ln \frac{b}{a}\right) ^{\alpha }\right\}
+\dint\limits_{x}^{C}\left( \ln \frac{t}{a}\right) ^{\alpha
-1}dt-\dint\limits_{a}^{x}\left( \ln \frac{t}{a}\right) ^{\alpha -1}dt
\end{eqnarray*}%
and%
\begin{eqnarray*}
&&\dint\limits_{C}^{b}\frac{\left\vert y-t\right\vert }{t}\left( \ln \frac{b%
}{t}\right) ^{\alpha -1}dt \\
&=&\frac{y}{\alpha }\left\{ \left( 1-\lambda \right) ^{\alpha }\left( \ln 
\frac{b}{a}\right) ^{\alpha }-2\left( \ln \frac{b}{y}\right) ^{\alpha
}\right\} +\dint\limits_{y}^{b}\left( \ln \frac{b}{t}\right) ^{\alpha
-1}dt-\dint\limits_{C}^{y}\left( \ln \frac{b}{t}\right) ^{\alpha -1}dt.
\end{eqnarray*}%
(3) If $a\leq x\leq y\leq C\leq b$, then%
\begin{eqnarray*}
&&\dint\limits_{a}^{C}\frac{\left\vert x-t\right\vert }{t}\left( \ln \frac{t%
}{a}\right) ^{\alpha -1}dt \\
&=&\frac{x}{\alpha }\left\{ 2\left( \ln \frac{x}{a}\right) ^{\alpha
}-\lambda ^{\alpha }\left( \ln \frac{b}{a}\right) ^{\alpha }\right\}
+\dint\limits_{x}^{C}\left( \ln \frac{t}{a}\right) ^{\alpha
-1}dt-\dint\limits_{a}^{x}\left( \ln \frac{t}{a}\right) ^{\alpha -1}dt
\end{eqnarray*}%
and%
\begin{equation*}
\dint\limits_{C}^{b}\frac{\left\vert y-t\right\vert }{t}\left( \ln \frac{b}{t%
}\right) ^{\alpha -1}dt=\dint\limits_{C}^{b}\left( \ln \frac{b}{t}\right)
^{\alpha -1}dt-\frac{y}{\alpha }\left( 1-\lambda \right) ^{\alpha }\left(
\ln \frac{b}{a}\right) ^{\alpha }.
\end{equation*}%
Using the inequality (\ref{2-2a}) and the above identities $\int_{a}^{C}%
\frac{\left\vert x-t\right\vert }{t}\left( \ln \frac{t}{a}\right) ^{\alpha
-1}dt$ and $\int_{V}^{b}\frac{\left\vert y-t\right\vert }{t}\left( \ln \frac{%
b}{t}\right) ^{\alpha -1}dt$, we derive the inequality (\ref{2-2}). This
completes the proof.
\end{proof}

Under the assumptions of Theorem \ref{2.2}, we have the following
corollaries and remarks:

\begin{corollary}
In Theorem \ref{2.2}, if we take $\alpha =1$, then the inequality (\ref{2-2}%
) reduces the following inequality:%
\begin{equation*}
\left\vert \left( 1-\lambda \right) f(x)+\lambda f(y)-\frac{1}{\ln b-\ln a}%
\dint\limits_{a}^{b}\frac{f(t)}{t}dt\right\vert \leq \frac{MC_{1,\lambda
}(x,y)}{\ln \frac{b}{a}}
\end{equation*}
\end{corollary}

\begin{corollary}
In Theorem \ref{2.2}, let $\delta \in \left[ \frac{1}{2},1\right] $, $%
x=a^{\delta }b^{1-\delta }$ and $y=a^{1-\delta }b^{\delta }$. Then, we have
the inequality%
\begin{equation*}
\left\vert \lambda ^{\alpha }f(a^{\delta }b^{1-\delta })+(1-\lambda
)^{\alpha }f(a^{1-\delta }b^{\delta })-\frac{\Gamma (\alpha +1)}{\left( \ln 
\frac{b}{a}\right) ^{\alpha }}\left[ J_{C-}^{\alpha }f(a)+J_{C+}^{\alpha
}f(b)\right] \right\vert
\end{equation*}%
\begin{equation}
\leq ML(\alpha ,\lambda ,\delta )  \label{2-1b}
\end{equation}%
where
\end{corollary}

(i) If $\lambda \leq 1-\delta $, then%
\begin{eqnarray*}
L(\alpha ,\lambda ,\delta ) &=&a^{\delta }b^{1-\delta }\lambda ^{\alpha
}+a^{1-\delta }b^{\delta }\left[ \left( 1-\lambda \right) ^{\alpha }-2\left(
1-\delta \right) ^{\alpha }\right] \\
&&+\frac{\alpha }{\left( \ln \frac{b}{a}\right) ^{\alpha }}\left\{
\dint\limits_{a^{1-\delta }b^{\delta }}^{b}\left( \ln \frac{b}{t}\right)
^{\alpha -1}dt-\dint\limits_{C}^{a^{1-\delta }b^{\delta }}\left( \ln \frac{b%
}{t}\right) ^{\alpha -1}dt-\dint\limits_{a}^{C}\left( \ln \frac{t}{a}\right)
^{\alpha -1}dt\right\}
\end{eqnarray*}

(ii) If $1-\delta \leq \lambda \leq \delta $, then%
\begin{eqnarray*}
&&L(\alpha ,\lambda ,\delta ) \\
&=&a^{\delta }b^{1-\delta }\left[ 2\left( 1-\delta \right) ^{\alpha
}-\lambda ^{\alpha }\right] +a^{1-\delta }b^{\delta }\left[ \left( 1-\lambda
\right) ^{\alpha }-2\left( 1-\delta \right) ^{\alpha }\right] +\frac{\alpha 
}{\left( \ln \frac{b}{a}\right) ^{\alpha }} \\
&&\times \left\{ \dint\limits_{a^{\delta }b^{1-\delta }}^{C}\left( \ln \frac{%
t}{a}\right) ^{\alpha -1}dt-\dint\limits_{a}^{a^{\delta }b^{1-\delta
}}\left( \ln \frac{t}{a}\right) ^{\alpha -1}dt+\dint\limits_{a^{1-\delta
}b^{\delta }}^{b}\left( \ln \frac{b}{t}\right) ^{\alpha
-1}dt-\dint\limits_{C}^{a^{1-\delta }b^{\delta }}\left( \ln \frac{b}{t}%
\right) ^{\alpha -1}dt\right\} .
\end{eqnarray*}

(iii) If $\delta \leq \lambda $, then%
\begin{eqnarray*}
L(\alpha ,\lambda ,\delta ) &=&a^{\delta }b^{1-\delta }\left[ 2\left(
1-\delta \right) ^{\alpha }-\lambda ^{\alpha }\right] -a^{1-\delta
}b^{\delta }\left( 1-\lambda \right) ^{\alpha } \\
&&+\frac{\alpha }{\left( \ln \frac{b}{a}\right) ^{\alpha }}\left\{
\dint\limits_{C}^{b}\left( \ln \frac{b}{t}\right) ^{\alpha
-1}dt+\dint\limits_{a^{\delta }b^{1-\delta }}^{C}\left( \ln \frac{t}{a}%
\right) ^{\alpha -1}dt-\dint\limits_{a}^{a^{\delta }b^{1-\delta }}\left( \ln 
\frac{t}{a}\right) ^{\alpha -1}dt\right\} .
\end{eqnarray*}

\begin{corollary}
In Theorem \ref{2.2}, if we take $x=y=C$, then we have the inequality%
\begin{eqnarray}
&&\left\vert \left[ \lambda ^{\alpha }+(1-\lambda )^{\alpha }\right] f(x)-%
\frac{\Gamma (\alpha +1)}{\left( \ln \frac{b}{a}\right) ^{\alpha }}\left[
J_{x-}^{\alpha }f(a)+J_{x+}^{\alpha }f(b)\right] \right\vert  \label{2-1c} \\
&\leq &M\left\{ \left[ \lambda ^{\alpha }-\left( 1-\lambda \right) ^{\alpha }%
\right] x+\frac{\alpha }{\left( \ln \frac{b}{a}\right) ^{\alpha }}\left[
\dint\limits_{x}^{b}\left( \ln \frac{b}{t}\right) ^{\alpha
-1}dt-\dint\limits_{a}^{x}\left( \ln \frac{t}{a}\right) ^{\alpha -1}dt\right]
\right\} .  \notag
\end{eqnarray}
\end{corollary}

\begin{remark}
In the inequality (\ref{2-1c}), if we choose $\lambda =\frac{1}{2}$, then we
get the inequality (\ref{3-10}).
\end{remark}

\begin{corollary}
In the inequality (\ref{2-1b}), if we take $\delta =1$, then we have the
following weighted Hadamard-type inequalities for Lipschitzian functions via
Hadamard fractional integrals%
\begin{eqnarray}
&&\left\vert \lambda ^{\alpha }f(a)+(1-\lambda )^{\alpha }f(b)-\frac{\Gamma
(\alpha +1)}{\left( \ln \frac{b}{a}\right) ^{\alpha }}\left[ J_{C-}^{\alpha
}f(a)+J_{C+}^{\alpha }f(b)\right] \right\vert  \label{2-1d} \\
&\leq &M\left\{ b(1-\lambda )^{\alpha }-a\lambda ^{\alpha }+\frac{\alpha }{%
\left( \ln \frac{b}{a}\right) ^{\alpha }}\left[ \dint\limits_{a}^{C}\left(
\ln \frac{t}{a}\right) ^{\alpha -1}dt-\dint\limits_{C}^{b}\left( \ln \frac{b%
}{t}\right) ^{\alpha -1}dt\right] \right\} .  \notag
\end{eqnarray}
\end{corollary}

\begin{remark}
In the inequality (\ref{2-1d}), if we choose $\lambda =\frac{1}{2}$, then we
get the inequality (\ref{3-20}).
\end{remark}

\section{Application to Special Means}

Let us recall the following special means of two positive number $a,b$ with $%
b>a:$

\begin{enumerate}
\item The arithmetic mean%
\begin{equation*}
A=A\left( a,b\right) :=\frac{a+b}{2}.
\end{equation*}

\item The geometric mean%
\begin{equation*}
G=G\left( a,b\right) :=\frac{a+b}{2}.
\end{equation*}

\item The harmonic mean%
\begin{equation*}
H=H\left( a,b\right) :=\frac{2ab}{a+b}.
\end{equation*}

\item The logarithmic mean%
\begin{equation*}
L=L\left( a,b\right) :=\frac{b-a}{\ln b-\ln a}.
\end{equation*}

\item The identric mean%
\begin{equation*}
I=I(a,b)=\frac{1}{e}\left( \frac{b^{b}}{a^{a}}\right) ^{\frac{1}{b-a}}
\end{equation*}
\end{enumerate}

To prove the results of this section, we need the following lemma:

\begin{lemma}[\protect\cite{THH12}]
Let $f:\left[ a,b\right] \rightarrow 
\mathbb{R}
$ be differentiable with $\left\Vert f^{\prime }\right\Vert _{\infty
}<\infty .$ Then $f$ is an $M$-Lipschitzian function on $\left[ a,b\right] $
where $M=\left\Vert f^{\prime }\right\Vert _{\infty }.$
\end{lemma}

\begin{proposition}
\label{3.1}For $b>a>0$, $\lambda \in \left[ 0,1\right] $ and $n\geq 1$, we
have.%
\begin{eqnarray}
&&\left\vert \left( 1-\lambda \right) G^{n}(a,b)+\lambda A\left(
a^{n},b^{n}\right) -L\left( a^{n},b^{n}\right) \right\vert  \label{4-1} \\
&\leq &\frac{nb^{n-1}}{\ln b-\ln a}\left[ \frac{\lambda \left( b-a\right) }{2%
}\ln \frac{b}{a}+2\left( 1-2\lambda \right) \left( A\left( a,b\right)
-G\left( a,b\right) \right) \right] .  \notag
\end{eqnarray}
\end{proposition}

\begin{proof}
The proof follows by the inequality (\ref{3-5}) applied for the Lipschitzian
function $f(x)=x^{n}$ on $\left[ a,b\right] $.
\end{proof}

\begin{remark}
Let $\lambda =0$ and $\lambda =1$ in the inequality (\ref{4-1}). Then, using
the inequality (\ref{1-3}), we have the following inequalities respectively%
\begin{equation*}
0\leq L\left( a^{n},b^{n}\right) -G^{n}(a,b)\leq \frac{2nb^{n-1}}{\ln b-\ln a%
}\left( A\left( a,b\right) -G\left( a,b\right) \right) ,
\end{equation*}%
\begin{equation*}
0\leq A\left( a^{n},b^{n}\right) -L\left( a^{n},b^{n}\right) \leq \frac{%
nb^{n-1}}{\ln b-\ln a}\left[ \frac{b-a}{2}\ln \frac{b}{a}-2\left( A\left(
a,b\right) -G\left( a,b\right) \right) \right]
\end{equation*}
\end{remark}

\begin{proposition}
\label{3.2}For $b>a>0$ and $\lambda \in \left[ 0,1\right] ,$ we have%
\begin{eqnarray}
&&\left\vert \left( 1-\lambda \right) G(ae^{a},be^{b})+\lambda A\left(
ae^{a},be^{b}\right) -L\left( e^{a},e^{b}\right) L\left( a,b\right)
\right\vert  \label{4-2} \\
&\leq &\frac{e^{b}\left( b+1\right) }{\ln b-\ln a}\left[ \frac{\lambda
\left( b-a\right) }{2}\ln \frac{b}{a}+2\left( 1-2\lambda \right) \left(
A\left( a,b\right) -G\left( a,b\right) \right) \right] .  \notag
\end{eqnarray}
\end{proposition}

\begin{proof}
The proof follows by the inequality (\ref{3-5}) applied for the Lipschitzian
function $f(x)=xe^{x}$ on $\left[ a,b\right] $.
\end{proof}

\begin{remark}
Let $\lambda =0$ and $\lambda =1$ in the inequality (\ref{4-2}). Then, using
the inequality (\ref{1-3}), we have the following inequalities respectively%
\begin{equation*}
0\leq L\left( e^{a},e^{b}\right) L\left( a,b\right) -G(ae^{a},be^{b})\leq 
\frac{2e^{b}\left( b+1\right) }{\ln b-\ln a}\left( A\left( a,b\right)
-G\left( a,b\right) \right) ,
\end{equation*}%
\begin{eqnarray*}
0 &\leq &A\left( ae^{a},be^{b}\right) -L\left( e^{a},e^{b}\right) L\left(
a,b\right) \\
&\leq &\frac{e^{b}\left( b+1\right) }{\ln b-\ln a}\left[ \frac{b-a}{2}\ln 
\frac{b}{a}-2\left( A\left( a,b\right) -G\left( a,b\right) \right) \right]
\end{eqnarray*}
\end{remark}

\begin{proposition}
\label{3.3}For $b>a>0$, $\lambda \in \left[ 0,1\right] $ and $n\geq 1$, we
have.%
\begin{eqnarray}
&&\left\vert \left( 1-\lambda \right) G^{-1}(a,b)+\lambda H^{-1}\left(
a,b\right) -L\left( a,b\right) G^{-2}(a,b)\right\vert  \label{4-3} \\
&\leq &\frac{1}{a^{2}\left( \ln b-\ln a\right) }\left[ \frac{\lambda \left(
b-a\right) }{2}\ln \frac{b}{a}+2\left( 1-2\lambda \right) \left( A\left(
a,b\right) -G\left( a,b\right) \right) \right] .  \notag
\end{eqnarray}
\end{proposition}

\begin{proof}
The proof follows by the inequality (\ref{3-5}) applied for the Lipschitzian
function $f(x)=1/x$ on $\left[ a,b\right] $.
\end{proof}

\begin{remark}
Let $\lambda =0$ and $\lambda =1$ in the inequality (\ref{4-3}). Then, using
the inequality (\ref{1-3}), we have the following inequalities respectively%
\begin{equation*}
0\leq L\left( a,b\right) -G(a,b)\leq \frac{G^{2}(a,b)}{a^{2}\left( \ln b-\ln
a\right) }\left( A\left( a,b\right) -G\left( a,b\right) \right) ,
\end{equation*}%
\begin{eqnarray*}
0 &\leq &G^{2}(a,b)-L\left( a,b\right) H\left( a,b\right) \\
&\leq &\frac{G^{2}(a,b)H\left( a,b\right) }{a^{2}\left( \ln b-\ln a\right) }%
\left[ \frac{b-a}{2}\ln \frac{b}{a}-2\left( A\left( a,b\right) -G\left(
a,b\right) \right) \right]
\end{eqnarray*}
\end{remark}

\begin{proposition}
For $b>a>0$ and $\lambda \in \left[ 0,1\right] ,$ we have%
\begin{equation*}
\ln G(a,b)\leq \frac{1}{a\left( \ln b-\ln a\right) }\left[ \frac{\lambda
\left( b-a\right) }{2}\ln \frac{b}{a}+2\left( 1-2\lambda \right) \left(
A\left( a,b\right) -G\left( a,b\right) \right) \right]
\end{equation*}
\end{proposition}

\begin{proof}
The proof follows by the inequality (\ref{3-5}) applied for the Lipschitzian
function $f(x)=\ln x$ on $\left[ a,b\right] $.
\end{proof}

\begin{proposition}
For $b>a>e^{-1}$ and $\lambda \in \left[ 0,1\right] ,$ we have%
\begin{eqnarray*}
&&\left\vert \left( 1-\lambda \right) G(a,b)\ln G(a,b)+\lambda \ln
G(a^{a},b^{b})-L\left( a,b\right) \ln I(a,b)\right\vert \\
&\leq &\frac{1+\ln b}{\ln b-\ln a}\left[ \frac{\lambda \left( b-a\right) }{2}%
\ln \frac{b}{a}+2\left( 1-2\lambda \right) \left( A\left( a,b\right)
-G\left( a,b\right) \right) \right]
\end{eqnarray*}
\end{proposition}

\begin{proof}
The proof follows by the inequality (\ref{3-5}) applied for the Lipschitzian
function $f(x)=x\ln x$ on $\left[ a,b\right] $.
\end{proof}

\end{document}